\documentclass[11pt,a4paper,reqno]{amsart}
\usepackage{graphicx}
\usepackage[linktocpage=true,colorlinks,citecolor=blue,linkcolor=blue,urlcolor=blue]{hyperref}
\usepackage{amssymb}
\usepackage{cite}
\usepackage{amsmath}
\usepackage{latexsym}
\usepackage{amscd}
\usepackage{amsthm}
\usepackage{mathrsfs}
\usepackage{url}
\usepackage[utf8]{inputenc}
\usepackage[english]{babel}
\usepackage{amsfonts}

\numberwithin{equation}{section}
\def\R{\mathbb R}
\def\Z{\mathbb Z}

\def\d{\mathrm d}
\def\CA{\mathcal A}

\def\ee{\varepsilon}
\def\wt{\widetilde}
\def\gcd{\operatorname{gcd}}

\newtheorem{theorem}{Theorem}[section]
\newtheorem{lemma}[theorem]{Lemma}
\newtheorem{proposition}[theorem]{Proposition}

\newtheorem{corollary}[theorem]{Corollary}

\theoremstyle{remark}

\theoremstyle{definition}
\newtheorem{definition}[theorem]{Definition}

\theoremstyle{remark}

\numberwithin{equation}{section}

\begin{document}

\title{A robust version of Freiman's $3k-4$ Theorem and applications}

\author{Xuancheng Shao}
\thanks{X.S. was supported by a Glasstone Research Fellowship.}
\address{Department of Mathematics, University of Kentucky, Lexington, KY, 40506, USA}
\email{xuancheng.shao@uky.edu}

\author{Max Wenqiang Xu}
\thanks{M.W.X. was supported by a London Mathematics Society Undergraduate Research Bursary and the Mathematical Institute at University of Oxford.}
\address{Department of Mathematics, University College London, Gower Street, London, WC1E 6BT, United Kingdom}
\email{wenqiang.xu@ucl.ac.uk}


\maketitle

\begin{abstract}
We prove a robust version of Freiman's $3k - 4$ theorem on the restricted sumset $A+_{\Gamma}B$, which applies when the doubling constant is at most $\tfrac{3+\sqrt{5}}{2}$ in general and at most $3$ in the special case when $A = -B$. As applications, we derive robust results with other types of assumptions on popular sums, and structure theorems for sets satisfying almost equalities in discrete and continuous versions of the Riesz-Sobolev inequality.
\end{abstract}

\section{introduction}

For two sets $A, B \subset \Z$, the complete sumset $A+B$ is defined by
\[ A+B := \{a+b \colon a \in A, b \in B\}. \]
We have the trivial lower bound $|A+B| \geq |A| + |B| - 1$, and the equality holds if and only if $A, B$ are arithmetic progressions with the same common difference. Freiman's $3k-4$ theorem goes a step further, and states that if $|A| = |B| = k$ and $|A+B| \leq 3k-4$, then $A,B$ must have diameters bounded by $|A+B|-k+1$; i.e., there exist arithmetic progressions $P,Q$ with the same common difference and sizes at most $|A+B|-k+1$, such that $A \subset P$ and $B \subset Q$.

In this paper we aim at proving a robust version of Freiman's $3k-4$ theorem. For a subset $\Gamma \subset A \times B$, the restricted sumset $A+_{\Gamma}B$ is defined by
\[ A+_{\Gamma}B = \{a+b \colon (a,b) \in \Gamma\}. \]
We are mainly concerned with situations when $\Gamma$ is almost all of $A \times B$. The advantage of considering a robust version is to relax the rigid conditions and keep the important feature. See \cite{Tao} and \cite{BSG} for more work in this flavor.

\begin{theorem}\label{1+theta}
Let $\ee > 0$. Let $A,B \subset \Z$ be subsets with $|A|= |B|= n \ge \max\{3, 2\ee^{-1/2}\} $, and let $\Gamma \subset A \times B$ be a subset with $|\Gamma| \geq (1-\ee)n^2$. If $|A +_{\Gamma} B| < (1+\theta - 11\ee^{1/2})n$ where $\theta = \tfrac{\sqrt{5}+1}{2}$, then there are arithmetic progressions $P,Q$ with the same common difference and sizes at most $|A +_{\Gamma} B|- (1 - 5\ee^{1/2})n$, such that $|A \cap P| \geq (1 - \ee^{1/2}) n$ and $|B \cap Q| \geq (1 - \ee^{1/2})n$. 
\end{theorem}

Our primary focus is on the upper bound for $|A+_{\Gamma}B|$ in the assumption. While one naturally expects Theorem~\ref{1+theta} to hold with $1+\theta$ replaced by $3$, we are only able to prove this expected result in the special case when $B = -A$.

\begin{theorem}\label{best 3}
Let $\ee > 0$. Let $A \subset \Z$ be a subset with $|A|= n\ge \max\{3, 2\ee^{-1/2}\}$, and let $\Gamma \subset A \times A$ be a subset with $|\Gamma| \geq (1-\ee)n^2$. If $|A -_{\Gamma} A| < (3- 11 \ee^{1/2})n$, then there exists an arithmetic progression with length at most $|A -_{\Gamma} A| - (1 - 5\ee^{1/2})n$ such that $|A\cap P| \ge (1 -\ee^{1/2})n$.
\end{theorem}

As an application to Theorem~\ref{1+theta}, we prove a structure theorem for sets $A$ with the quantity
\[ C(A):= \frac{\# \{(a,a') \in A \times A : a + a ' \in A \}  }{|A|^2} \]
being close to its maximum value $3/4 + o(1)$. The fact that $C(A)$ is maximized when $A$ is a symmetric interval around the origin can be proved by a simple combinatorial argument; see~\cite[Proposition 2.1]{Carries}. The following result asserts that if $C(A)$ is close to $3/4$, then $A$ must be close to a symmetric arithmetic progression around the origin.

\begin{theorem}\label{carry}
Let $\ee \in (0, 2^{-50})$. Let $A \subset \mathbb{Z}$ be a subset with  $|A| = n \geq \ee^{-1}$ and $C(A) \ge 3/4 - \ee$. Then there exists an arithmetic progression $P$ containing $0$ and centered at $0$ such that  $|P|\le (1 + 280\ee^{1/4})n $ and $|P \cap A| \ge (1-10\ee^{1/4})n$.
\end{theorem}

We deduce Theorems~\ref{1+theta} and~\ref{best 3} in Section~\ref{sec:tech-statement} from two technical propositions, Proposition~\ref{main prop A+B} and Proposition~\ref{main prop A-A}. In Section~\ref{sec:abelian-group}, we develop robust analgoues of Kneser's theorem in abelian groups, which are used in Section~\ref{sec:proofs} to deduce the technical propositions. In Section~\ref{sec:other-robust}, we relate our main result with similar types of results in the literature~\cite{Matomaki,Mazur,Shao}. Section~\ref{sec:carry} contains the deduction of Theorem~\ref{carry} from Theorem~\ref{1+theta}, and finally in Section~\ref{sec:continuous-carry} we deduce a continuous version of Theorem~\ref{1+theta}, which is a special case of previous works on near equality in the Riesz-Sobolev inequality~\cite{Christ0,Christ1,Christ2,Christ3}.

\subsection*{Acknowledgement} We thank Ben Green for drawing our attention to the reference~\cite{Lev-2}.

\section{Technical version of the main theorems}\label{sec:tech-statement}

Before stating the technical versions of our main theorems, we need the following definition which also appeared in~\cite{Lev-theta}.

\begin{definition}
Let $A$, $B$ be two finite sets (in an ambient abelian group $G$), and let $\Gamma \subset A \times B$. We say that $A+_{\Gamma}B$ is $(K,s)$-regular if the following two conditions hold:
\begin{enumerate}
\item for each $a \in A$ there are at most $s$ values of $b\in B$ with $(a,b)\notin\Gamma$, and similarly for each $b \in B$ there are at most $s$ values of $a \in A$ with $(a,b) \notin\Gamma$;
\item if $r_{A+B}(x)\geq K$, then $x \in A+_{\Gamma}B$.
\end{enumerate}
\end{definition}

Here $r_{A+B}(x)$ denotes the number of ways to write $x=a+b$ with $a \in A$ and $b \in B$. 
 
\begin{proposition}\label{main prop A+B}
Let $\ell$ be a positive integer and let $A, B \subset \{0,1,\cdots,\ell\}$ be subsets with size $ |A|= |B| = n \ge 3$. Assume that $0,\ell\in A$, $0 \in B$ and $\gcd (A \cup B)= 1$. Let $\Gamma \subset A \times B$ be a subset such that $A+_{\Gamma}B$ is $(K, s)$-regular for some $K \geq 2$ and $s \geq 0$.
 Then 
\[  |A +_\Gamma B | \geq \begin{cases} \ell +n -2s & \text{if }\ell < 2n -2K -2, \\  (\theta + 1)n - 4s -2 K - \theta & \text{if }\ell \ge 2n -2K -2,\end{cases} \] 
where $\theta = \frac{1 +\sqrt{5}}{2}.$
\end{proposition}

In the special case when $A = B$, this becomes~\cite[Theorem 1]{Lev-theta}. In fact, our proof will follow arguments from~\cite{Lev-theta} with modifications to accommodate the asymmetric nature of our result. While Proposition~\ref{main prop A+B} is the technical version of Theorem~\ref{1+theta}, the following proposition corresponds to Theorem~\ref{best 3}.

\begin{proposition}\label{main prop A-A}
Let $\ell$ be a positive integer and let $A \subset \{0,1,\cdots,\ell\}$ be a subset with size $|A|=n$. Assume that $0,\ell \in A$ and  $\gcd(A)=1$. Let $\Gamma \subset A \times A$ be a subset such that $A-_{\Gamma}A$ is $(K,s)$-regular for some $K\geq 2$ and  $s \geq 0$.
Then
\[  |A-_{\Gamma}A| \geq \begin{cases} \ell + n - 2s & \text{if }\ell < 2n - 2K -2, \\ 3n - 4s - 2K -2 & \text{if }\ell \geq 2n - 2K - 2. \end{cases} \]
\end{proposition}

The reason that we are able to get a better result in the case when $A = -B$ stems from a special combinatorial argument (see Proposition~\ref{kneser 3} below) that is not available in the general setting.

We will prove Propositions~\ref{main prop A+B} and~\ref{main prop A-A} in Section~\ref{sec:proofs}, after developing the essential ingredient in Section~\ref{sec:abelian-group} about restricted sumsets in abelian groups. In the remainder of this section, we deduce Theorems~\ref{1+theta} and~\ref{best 3} from Propositions~\ref{main prop A+B} and~\ref{main prop A-A}, respectively.

\begin{proof}[Proof of Theorem~\ref{1+theta} assuming Proposition~\ref{main prop A+B}]
Since $ |\Gamma |\ge (1- \ee)n^2$, for every $a \in A$ outside an exceptional set of size at most $\ee^{1/2}n$, we have $(a,b) \in \Gamma$ for all but at most $\ee^{1/2}n$ values of $B$. Thus we can find $A' \subset A$ and $B' \subset B$ such that $|A'| = |B'| \ge (1- \ee^{1/2})n$ and for any $a' \in A'$ we have 
\[\# \{b'\in B' \colon (a',b') \notin \Gamma\} \le \ee^{1/2}n.\]
Let $\Gamma^* = (A' \times B') \cap \Gamma$. Set
$s = \ee^{1/2}n$ and $K = \ee^{1/2}n$, and take
\[ \Gamma' = \Gamma^* \cup \{(a',b') \in A'\times B' \colon r_{A'+B'}(a'+b') \geq K\}. \]
Then by construction $A'+_{\Gamma'}B'$ is $(K,s)$-regular. After translating and dilating appropriately so that the smallest elements in $A'$ and $B'$ are both $0$ and $\text{gcd}(A' \cup B') = 1$,  let $\ell$ be the largest element in $A' \cup B'$. By Proposition~\ref{main prop A+B} we have
\[  |A' +_{\Gamma'} B' | \geq \begin{cases} \ell + |A'| - 2s & \text{if }\ell < 2|A'| - 2K - 2, \\  (\theta + 1)|A'| - 4s - 2K - \theta  & \text{if }\ell \ge 2|A'| - 2K - 2. \end{cases}
\] 
Since $|A'| \geq (1-\ee^{1/2})n$ and $n \geq 2\ee^{-1/2}$, the inequalities above lead to
\[ |A' +_{\Gamma'} B'| \geq \min(\ell + (1-3\ee^{1/2})n, (\theta+1-10\ee^{1/2})n). \]
Next we claim that
\[ |A+_{\Gamma}B| \geq |A'+_{\Gamma'}B'| - \ee^{1/2}n. \]
To see this, note that $A'+_{\Gamma^*}B' \subset A+_{\Gamma}B$, and that if $x \in (A'+_{\Gamma'}B') \setminus (A'+_{\Gamma^*}B')$ then $r_{A'+B'}(x) \geq K$ by the definition of $\Gamma'$, so that the number of such $x$ is at most
\[  \frac{|(A'\times B')\setminus \Gamma^*|}{K} \le \frac{|(A\times B)\setminus\Gamma|}{K} \leq \ee^{1/2}n.      \]
It follows that
\[ |A+_{\Gamma}B| \geq \min(\ell + (1-4\ee^{1/2})n, (\theta+1-11\ee^{1/2})n).  \]
By our assumption on the size of $A+_{\Gamma}B$, we must have
\[ |A+_{\Gamma}B| \geq \ell + (1-4\ee^{1/2})n. \]
Setting $P$ and $Q$ to be suitably translated and dilated copies of the interval $\{0,1,2,\cdots,\ell\}$ (so that $A' \subset P$ and $B' \subset Q$), we see that they have lengths
\[ \ell + 1 \leq |A+_{\Gamma}B| - (1-5\ee^{1/2})n, \]
and moreover $|A \cap P| \geq |A'| \geq (1-\ee^{1/2})n$ and similarly $|B \cap Q| \geq |B'| \geq (1-\ee^{1/2})n$. This completes the proof. 
\end{proof}

\begin{proof}[Proof of Theorem~\ref{best 3} assuming Proposition~\ref{main prop A-A}]
The proof of Theorem~\ref{best 3} is almost the same as the proof of Theorem~\ref{1+theta}. The only difference is that we apply Proposition~\ref{main prop A-A} intead of Proposition~\ref{main prop A+B} in the final step to complete our proof. 
\end{proof}

\section{Robust results in the abelian group setting}\label{sec:abelian-group}

In this section we prove robust analogues of Kneser's theorem in abelian groups. Our first result is an extension of~\cite[Theorem 2]{Lev-theta} to two distinct sets $A,B$.

\begin{proposition}\label{kneser theta}
Let $G$ be an abelian group and let $A, B\subset G$ be subsets with $n = |A| \le |B|$. Let $\Gamma \subset A \times B$ be a subset such that $A+_{\Gamma}B$ is $(K,s)$-regular for some $K,s \geq 0$. Suppose that $A +_\Gamma B \neq A + B$. Then 
\[ |A + _\Gamma B| > \theta n - K -2s,  \]
where $\theta = \tfrac{1 + \sqrt{5}}{2}$.     
\end{proposition}

\begin{proof}
For the purpose of contradiction, assume that $|A+_\Gamma B| \le \theta n -K -2s$. The proof is split into the following four steps.

\subsection*{Step 1}
By the pigeonhole principle, there exists $\sigma \in G$ such that
\[   r_{A + B}(\sigma) \geq \frac{|\Gamma|}{|A+_{\Gamma}B|} > \frac {n^2-ns}{\theta n - K - 2s} \ge  \frac{n(n- s)}{\theta (n-s)}= (\theta - 1)         n. \] 
From now on fix such $\sigma$.

\subsection*{Step 2} 
We show that $r_{A-A}(c) \geq (2-\theta)n + K$ for each $c \in B-B$. Take $c = b'-b \in B-B$ for some $b,b' \in A$. Since $A+_{\Gamma}B$ is $(K,s)$-regular, we have
\[ |(A+b) \backslash ( A +_\Gamma B)| \le s, ~~~ |(A+b') \backslash (A +_\Gamma B)| \le s.   \]
It follows that
\[ |(A+b) \cup (A+b')| \le |A +_\Gamma B| +2s \le \theta n - K,                   \]
and thus
\[ |(A+b) \cap (A+b') | \ge 2|A| - \theta n + K = (2-\theta)n + K.                              \]
Note that each element in the intersection $(A+b) \cap (A+b')$ is of the form $a + b = a' + b'$ for some $a,a' \in A$, and this leads to a representation 
\[ c = b'-b = a - a'.         \]
Hence $r_{A-A}(c)\ge(2-\theta)n + K$ as desired. Similarly one can also show that $r_{B-B}(d) \geq 2|B|-\theta n + K$ for each $d \in A-A$.

\subsection*{Step 3} 
We claim that for the above fixed $\sigma$ and each $c \in B - B$, there are at least $K$ triples $(b, a, a') \in B \times A \times A$ such that $c = a' - a$ and $\sigma = a + b$. To see this, note that the number of $(a,a') \in A \times A$ with $c = a' - a$ is at least $(2-\theta)n+K$ by Step~2, and the number of $(a,b) \in A \times B$ with $\sigma = a+b$ is at least $(\theta-1)n$ by Step~1. Hence the number of such triples $(b,a,a')$ is at least
\[ (2-\theta)n + K + (\theta-1)n - |A| \geq K, \]
as claimed. For any such triple we have $c + \sigma = a' + b \in A+B$, and thus from the $(K,s)$-regularity assumption we deduce that $c + \sigma \in A +_\Gamma B$. This shows that $B - B + \sigma \subset A +_\Gamma B$. Similarly, for each $d \in A-A$ the number of triples $(b,b',a) \in B \times B \times A$ with $d = b'-b$ and $\sigma = a+b$ is at least
\[ 2|B| - \theta n + K + (\theta-1)n - |B| \geq K, \]
showing that $A-A+\sigma \subset A+_{\Gamma}B$ as well. Hence
\begin{equation} \label{contradiction}
|(A - A) \cup (B - B)| \le  | A +_\Gamma B |.
\end{equation}

\subsection*{Step 4}
Finally, we show that \eqref{contradiction}  is impossible. Since $A +_\Gamma B \neq A + B$, we may find $a \in A$ and $b \in B$ such that $ a + b \notin A +_\Gamma B$. Thus $r_{A+B}(a+b) < K$, which implies that 
\[ |(-a + A) \cap (b - B)| < K.                     \]  
Therefore, 
\[
\begin{split} 
|(A-A) \cup (B-B)| &\ge |(-a + A) \cup (b - B)| > 2n - K \\
&> \theta n - K - 2s \ge |A +_\Gamma B|.               
\end{split}
\]
This contradicts~\eqref{contradiction} and the proof is completed.
\end{proof}

Motivated by arguments from~\cite[Section 4]{Lev-2}, we can obtain a better bound in the special case when $B = -A$.

\begin{proposition}\label{kneser 3}
Let $G$ be an abelian group and let $A \subset G$ be a subset with $|A| = n$. Let $\Gamma \subset A \times A$ be a subset such that $A-_{\Gamma}A$ is $(K,s)$-regular for some $K,s \geq 0$. Suppose that $A-_{\Gamma}A \neq A-A$. Then
\[ |A-_{\Gamma}A| > 2n - K - 2s. \]
\end{proposition}

\begin{proof}
Since $A-_{\Gamma}A \neq A-A$, there exists $(a, -a') \in A \times -A$ such that $a-a'\notin A-_{\Gamma} A$.  This implies that $r_{A-A}(a-a') < K$ by the definition of $(K,s)$-regular. Let $N(a)$ be the set of neighbours of $a$ in $\Gamma$:
\[ N(a) = \{n_a \in -A \colon (a, n_a) \in \Gamma\}, \]
and define $N(a')$ similarly. Since
\[|(a + N(a)) \cap (a' + N(a')| \leq r_{A-A}(a-a') < K, \]
it follows that
\[ |(a + N(a)) \cup (a' + N(a'))| > |N(a)| + |N(a')| - K \ge 2n - 2s -K.             \]
The proof is completed by noting that $(a + N(a)) \cup (a'- N(a')) \subset A-_{\Gamma}A$.
\end{proof}

\section{Proof of Propositions~\ref{main prop A+B} and~\ref{main prop A-A}}\label{sec:proofs}

In this section we prove Propositions~\ref{main prop A+B} and~\ref{main prop A-A}. The proofs roughly follow Lev and Smeliansky's proof~\cite{LS} of Freiman's $3k-4$ theorem by reducing modulo $\ell$ and applying Kenser's theorem in $\Z/\ell\Z$, with a significant amount of work devoted to analyzing the stabilizer $H$ from Kneser's theorem.

\begin{proof}[Proof of Proposition~\ref{main prop A+B}]

Let $\widetilde{A}, \wt{B} \subset \Z/\ell\Z$ be the images of $A, B$ modulo $\ell$, respectively, so that $|\widetilde{A}| = n-1$ and $|\wt{B}| \geq n-1$. Let $\widetilde{\Gamma} \subset \widetilde{A} \times \widetilde{B}$ be the image of $\Gamma$ modulo $\ell$. The $(K,s)$-regularity of $A+_{\Gamma}B$ implies that $\wt{A}+_{\wt{\Gamma}}\wt{B}$ is $(2K,s)$-regular:
\begin{enumerate}
\item for any $\widetilde{a} \in \widetilde{A}$ there are at most $s$ values of $\widetilde{b} \in \widetilde{B}$ such that $(\widetilde{a},\wt{b}) \notin \wt{\Gamma}$, and similarly for any $\widetilde{b} \in \widetilde{B}$ there are at most $s$ values of $\widetilde{a} \in \widetilde{A}$ such that $(\widetilde{a},\wt{b}) \notin \wt{\Gamma}$; 
\item if $r_{\wt{A}+\wt{B}}(\wt{x}) \geq 2K$, then $\wt{x} \in \wt{A}+_{\wt{\Gamma}}\wt{B}$.
\end{enumerate}
Indeed, (1) holds because each bad pair $(a,b)\notin \Gamma$ reduces to at most one bad pair $(\wt{a},\wt{b}) \notin \wt{\Gamma}$. To show (2), take $\wt{x} \notin \wt{A}+_{\wt{\Gamma}}\wt{B}$, so that any pre-image of $\wt{x}$ is not in $A+_{\Gamma} B$. If $\wt{x} = 0$, then 
\[ r_{\wt{A}+_{\wt{\Gamma}}\wt{B}}(0)\le r_{A +_\Gamma B}(0) + r_{A +_\Gamma B}(\ell) + r_{A +_\Gamma B}(2 \ell) \le 1 + (K-1) + 1 < 2K, \] 
since $K\ge2$. If $\wt{x} \neq 0$, then it has two possible pre-images $x,x+\ell$ and thus
\[ r_{\wt{A}+_{\wt{\Gamma}}\wt{B}}(\wt{x}) \le r_{A +_\Gamma B}(x) + r_{A +_\Gamma B}(x + \ell) < K +K =2K. \]
This proves the $(2K,s)$-regularity of $\wt{A}+_{\wt{\Gamma}}\wt{B}$.

Note that if $(0,b)$ and $(\ell,b)$ both lie in $\Gamma$ for some $b \in B$, then both $0+b$ and $\ell+ b$ lie in $A+_{\Gamma}B$, but they reduce to the same element modulo $\ell$ and are thus counted only once in $\widetilde{A}+_{\widetilde{\Gamma}} \widetilde{B}$. Since there are at least $n-2s$ values of such $b \in B$, we conclude that
\begin{equation}\label{eq:A-to-Atilde} 
|A +_{\Gamma} B|  \ge  | \widetilde{A}+_{\widetilde{\Gamma}} \widetilde{B} | + n - 2s. 
\end{equation}
If $\ell < 2n -2K-2$, then $\ell + 2K < 2n - 2 \le |\wt{A}| + |\wt{B}|$. It follows that 
\[ |\wt{A} \cap (x-\wt{B})| \geq |\wt{A}| + |\wt{B}| - \ell > 2K \]
for any $x \in \Z/\ell\Z$, and hence $r_{\wt{A}+\wt{B}}(x) > 2K $ for any $x$. Since $\wt{A}+_{\wt{\Gamma}}\wt{B}$ is $(2K,s)$-regular, we have $\wt{A}+_{\wt{\Gamma}}\wt{B} = \Z/ \ell \Z$, and by~\eqref{eq:A-to-Atilde} we have $|A +_{\Gamma} B|\ge \ell +n -2s$, as desired.

If $\ell \ge 2n -2K -2$ and $\widetilde{A}+_{\widetilde{\Gamma}} \widetilde{B} \neq  \widetilde{A}+ \widetilde{B}$,  then we may apply Proposition~\ref{kneser theta} to get
\[ |\wt{A}+_{\wt{\Gamma}} \wt{B} | > \theta(n-1) - 2K - 2s. \]
Thus by~\eqref{eq:A-to-Atilde} we get the desired bound 
\[
|A +_{\Gamma} B| \ge (\theta + 1)n - 2K -4s -\theta.       
\]

It remains to consider the case when $\ell \ge 2n -2K-2$ and $\widetilde{A}+_{\widetilde{\Gamma}} \widetilde{B} =  \widetilde{A}+ \widetilde{B}$.
We may assume that
\begin{equation}\label{eq:assumption} 
 |\widetilde{A} + \widetilde{B}| = |\widetilde{A} +_{\widetilde{\Gamma}} \widetilde{B}|\le 2n -2s-4,
\end{equation}
since otherwise the conclusion follows from~\eqref{eq:A-to-Atilde}. Let $H = H (\widetilde{A} + \widetilde{B})$ be the period of $\widetilde{A} + \widetilde{B}$, so that $\wt{A} + \wt{B} + H = \wt{A} + \wt{B}$. By Kneser's theorem applied to $\widetilde{A} + \widetilde{B}$, we have
\begin{equation}\label{eq:Kneser} 
|\widetilde{A} + \widetilde{B}| \geq |\wt{A} + H| + |\wt{B} + H| - |H|. 
\end{equation}

To count the number of elements $x \in A+_{\Gamma}B$, we divide into three cases according to whether $\wt{x}$, the projection of $x$ modulo $\ell$, lies in $\wt{B}$, $(\wt{B}+H)\setminus\wt{B}$, or $(\wt{A}+\wt{B})\setminus (\wt{B}+H)$. Note that $\wt{B} + H \subset \wt{A} + \wt{B}$ since $\wt{A} + \wt{B} + H = \wt{A} + \wt{B}$ and $0 \in \wt{A}$. Note also that if $\wt{A} \subset H$ then one only needs to consider the first two cases as the third set is empty.

\subsection*{Case 1} 

First we claim that
\[ \# \{ x\in A + _\Gamma B: \widetilde{x} \in \wt{B}\} \ge 2n-2s-1. \]
Write $B= \{b_1, b_2, \cdots, b_n \}$ where $0=b_1 < b_2 < \cdots < b_n \le \ell$. Since $0,\ell \in A$, the following $2n-1$ sums all lie in $A+B$:
\[   0 + b_1, 0 + b_2,\cdots,  0 + b_n, \ell + b_2, \ell + b_3,\cdots, \ell + b_n.  \]
Moreover, at most $2s$ of them do not belong to $A+_\Gamma B$, and the claim follows.

\subsection*{Case 2}
We claim that
\[ \# \{ x\in A + _\Gamma B:\wt{x} \in (\wt{B}+H)\setminus\wt{B}\} \ge |(\wt{B} + H)\backslash \wt{B}| = |\wt{B} + H| - |\wt{B}|. \] 
Indeed, since $(\wt{B} + H)\backslash \wt{B} \subset \wt{A} + \wt{B} = \wt{A} +_{\widetilde{\Gamma}} \widetilde{B}$, every $\wt{x} \in (\wt{B} + H)\backslash \wt{B}$ has a pre-image in $A+_\Gamma B$.

We may now conclude the proof when $\wt{A} \subset H$. If $H = \Z/\ell\Z$, then from~\eqref{eq:Kneser} it follows that $|\wt{A}+_{\wt{\Gamma}}\wt{B}| = |\wt{A}+\wt{B}| \geq \ell$. Hence by~\eqref{eq:A-to-Atilde}, we have
\[ |A+_{\Gamma}B| \geq \ell + n - 2s \geq 3n-2K-2s-2 > (\theta+1)n - 4s - 2K - \theta, \]
as desired. If $H \subsetneq \Z/\ell\Z$, then $\wt{B}+H$ must contain at least two cosets of $H$, since otherwise $\wt{B}$ is contained in a single coset of $H$, contradicting the assumptions that $0 \in \wt{B}$ and $\text{gcd}(A \cup B)=1$. Hence $|\wt{B}+H| \geq 2|H| \geq 2|\wt{A}| = 2n-2$. Combining the bounds from Cases 1 and 2, we have
\[
\begin{split} 
|A+_{\Gamma}B| &\geq 2n-2s-1 + |\wt{B}+H|-|\wt{B}| \geq 2n-2s-1+ (2n-2) - n \\
&> (\theta+1)n-4s-2K-\theta,
\end{split}
\]
as desired. From now on assume that $\wt{A} \subsetneq H$, and we will need to consider Case 3 as well.

\subsection*{Case 3}

Finally we prove that
\[ \# \{ x\in A + _\Gamma B: \wt{x} \in (\wt{A}+\wt{B})\setminus (\wt{B}+H)\} \ge 2|\wt{B}| - |\wt{B} + H| -2s -1. \]
Since $\wt{A} + \wt{B} + H = \wt{A} + \wt{B}$, the set $(\wt{A}+\wt{B})\setminus (\wt{B}+H)$ is a disjoint union of $N$ cosets of $H$, where
\[ N = \frac{|\wt{A}+\wt{B}| - |\wt{B}+H|}{|H|}. \]
Note that $N \geq 1$, since otherwise $|\wt{A} + \wt{B} | = |\wt{B}+H|$, and by~\eqref{eq:Kneser} we get $|\wt{A}+H| \leq |H|$, contradicting the assumption that $\wt{A} \subsetneq H$. Fix a coset $\wt{a}+\wt{b} + H$ in $(\wt{A}+\wt{B})\setminus (\wt{B}+H)$, and  let us count the quantity 
\[M:=\#\{ x \in A+_{\Gamma}B:   \wt{x} \in \wt{a}+\wt{b}+H \}.  \]
Define $A'$ to be the set of those elements in $A$ whose projections modulo $\ell $ lie in the coset $\wt{a}+H$, and similarly define $B'$ to be the set of those elements in $B$ whose projections modulo $\ell$ lie in $\wt{b}+H$. Since any sum in $A'+_{\Gamma}B'$ projects to an element in the desired coset $\wt{a}+\wt{b}+H$, we have
\begin{equation}\label{eq:Mbound} 
M \geq |A'+_{\Gamma}B'| \geq |A'| + |B'| - 2s -1,
\end{equation}
where the second inequality follows by writing 
\[ A' = \{x_1 < x_2 < \cdots < x_{|A'|}\}, \ \ B' = \{y_1 < y_2 < \cdots < y_{|B'|}\}, \]
and noting that among the $|A'|+|B'|-1$ distinct sums
\[ x_1+y_1, x_1+y_2, \cdots, x_1+y_{|B'|},x_2+y_{|B'|}, \cdots, x_{|A'|}+y_{|B'|}, \]
at most $2s$ of them do not lie in $A'+_{\Gamma}B'$.

Let $\wt{A}'$ be the image of $A'$ after modulo $\ell$; i.e. $\wt{A}' = \wt{A} \cap (\wt{a}+H)$. Since $(\wt{a}+H) \setminus \wt{A} = (\wt{a}+H) \setminus \wt{A}'$ is contained in $(\wt{A}+H) \setminus \wt{A}$, we have
\[ |\wt{A}+H| - |\wt{A}| \geq |\wt{a}+H| - |\wt{A}'| \geq |H| - |A'|. \]
It follows that
\[|A'|\ge  |H| + |\wt{A}| - |\wt{A}  + H|, \]
and similarly,
\[|B'|\ge |H| + |\wt{B}| - |\wt{B} + H|.\]
Combining the two bounds above with~\eqref{eq:Mbound} and~\eqref{eq:Kneser}, we obtain
\[
\begin{split} 
M & \ge 2|H| + |\wt{A}| + |\wt{B}| - |\wt{A} +H | - |\wt{B} + H | -2s-1 \\
&\ge |H| + |\wt{A}| + |\wt{B}| - |\wt{A} +\wt{B}| -2s -1.
\end{split}
\]
Call the lower bound above $M_0$ so that $M \geq M_0$. In particular, we have $M_0 \geq |H|+1$ by~\eqref{eq:assumption}. We conclude that the total number of $x \in A+_{\Gamma}B$ counted in this case is at least $M_0N$, and
\[ 
\begin{split}
M_0N &= N|H| + N (M_0 - |H|) \geq N|H| + M_0 - |H| \\ 
& = |\wt{A}+\wt{B}| - |\wt{B}+H| + (|H| + |\wt{A}| + |\wt{B}| - |\wt{A} +\wt{B}| -2s -1 ) -|H|\\
& = |\wt{A}| + |\wt{B}| - |\wt{B} + H|-2s -1.
\end{split}
\]
Combining Cases 1, 2, and 3 together we get the desired bound
\[ 
\begin{split}
|A+_\Gamma B| & \ge  (2n-2s-1 )+   ( |\wt{B} + H| - |\wt{A}|) + (|\wt{A}| + |\wt{B}| - |\wt{B} + H|-2s -1 ) \\
& \ge 3n - 4s -3.
\end{split}
\]
This completes the proof.
\end{proof}

\begin{proof}[Proof of Proposition~\ref{main prop A-A}]

One can follow  the proof of Proposition~\ref{main prop A+B} above, taking $B = \ell - A$ and using Proposition~\ref{kneser 3} instead of Proposition~\ref{kneser theta}.
\end{proof}

\section{Robust results with other assumptions}\label{sec:other-robust}

Our main theorem deals with the restricted sumset $A+_{\Gamma}B$, where $\Gamma$ consists of almost all of the pairs from $A \times B$. One may consider other robust versions of the sumset $A+B$. In this section we deduce two such variants, and compare them with existing results in the literature.

\begin{corollary}\label{Pollard}
Let $A$ and $B$ be two sets of integers both with size $n \geq 3$. Let $0 < \delta < \tfrac{\sqrt{5}-1}{2}$. Assume that 
\[ \sum_{x} \min(1_A*1_B(x), t) \le (2 + \delta)nt, \ \ 1 \le t \le 2^{-10} \left(\tfrac{\sqrt{5}-1}{2}-\delta\right)^2 n. \]
Then there are arithmetic progressions $P,Q$ with the same common difference and sizes at most $(1+\delta)n + 10t^{1/2}n^{1/2}$, such that $|A \cap P| \geq n - 2t^{1/2}n^{1/2}$ and $|B \cap Q| \geq n - 2t^{1/2}n^{1/2}$. 
\end{corollary}

The assumption here is motivated by Pollard's theorem which states that
\[ \sum_x \min(1_A*1_B(x), t) \geq t(2n-t) \]
for any $0 \leq t \leq n$. In the case when $A=B$, this should be compared with~\cite[Theorem 1.1]{Mazur}, where the assumptions on $\delta$ and $t$ are:
\[ 0 < \delta < \tfrac{1}{4}, \ \ t \leq \tfrac{1}{5}\left(\tfrac{1}{4}-\delta\right)n. \] 
Thus Corollary~\ref{Pollard} extends the range of $\delta$, at the cost of requiring $t$ to be a much smaller constant times $n$.

\begin{proof}
Let $\Gamma$ be the set of pairs $(a,b) \in A\times B$ with $r_{A+B}(a+b) \geq t$. Then
\[ \sum_x \min(1_A*1_B(x), t) = \sum_{x\in A+_{\Gamma}B}t + \sum_{x\notin A+_{\Gamma}B} 1_A*1_B(x) = t|A+_{\Gamma}B| + |(A\times B)\setminus\Gamma|.  \] 
By assumption the quantity above is bounded by $(2+\delta)nt$. Thus it follows that
\[ |A+_{\Gamma}B| \leq (2+\delta)n \]
and
\[ |(A\times B)\setminus \Gamma| \leq (2+\delta)nt \leq 3nt. \]
We will apply Theorem~\ref{1+theta} with $\ee = 3t/n$. The upper bound on $t$ implies that
\[ 2+\delta \leq 1+\theta- 11 \ee^{1/2}, \]
where $\theta = \tfrac{\sqrt{5}+1}{2}$, so that the assumption of Theorem~\ref{1+theta} is satisfied. Thus there are arithmetic progressions $P,Q$ with the same common difference and sizes at most
\[ |A+_{\Gamma}B| - (1-5(3t/n)^{1/2})n \leq (1+\delta)n + 10t^{1/2}n^{1/2},
  \]
such that
\[ |A \cap P| \geq (1-(3t/n)^{1/2})n \geq n - 2t^{1/2}n^{1/2}, \ \ |B \cap Q| \geq n - 2t^{1/2}n^{1/2}. \]
This completes the proof.
\end{proof}

The second corollary deals with popular sums; see~\cite[Lemma 3.4]{Matomaki} and~\cite[Theorem 2.4]{Shao}.

\begin{corollary}\label{popular-sum}
Let $A$ and $B$ be two sets of integers both with size $n$. Let $\eta$ and $\delta$ be positive numbers such that $11 \eta^{1/2} + \delta < \frac{\sqrt{5}-1}{2}$ and $n > 2\eta^{-1/2}$. Assume that
\[ \#\left\{x \in A+B \colon r_{A+B}(x) \geq \frac{\eta n^2}{|A+B|} \right\} \leq (2 + \delta )n. \]
Then there exist arithmetic progressions $P,Q$ with the same common difference and sizes at most $(1 + \delta + 5 \eta^{1/2})n$, such that $|A \cap P| \geq (1 - \eta^{1/2}) n$ and $|B \cap Q| \geq (1 - \eta^{1/2})n$. 
\end{corollary}

Compared with~\cite[Lemma 3.4]{Matomaki} which was deduced from~\cite[Theorem 1]{Lev-theta}, Corollary~\ref{popular-sum} is applicable to two distinct sets $A,B$ and can be deduced from our Proposition~\ref{main prop A+B}.

\begin{proof}
In this proof, we write $r(x)$ for $r_{A+B}(x)$, and let $K = \eta n^2/|A+B|$. Let $\Gamma \subset A \times B$ be the set of $(a,b) \in A \times B$ with $r(a+b) \geq K$. From the inequality
\[ n^2 = \sum_{x} r(x) = \sum_{x\colon r(x) < K} r(x) + \sum_{x \colon r(x) \geq K} r(x) \leq K|A+B| + |\Gamma|, \]
it follows that
\[ |\Gamma| \geq n^2 - K|A+B| = (1-\eta)n^2. \]
We will apply Theorem~\ref{1+theta} with $\ee = \eta$. Since $2 + \delta < 1 + \theta - 11 \eta^{1/2}$ by assumption, we can apply Theorem~\ref{1+theta} to conclude that there are arithmetic progressions $P,Q$ with the same common difference and sizes at most
\[ (2+\delta)n - (1-5\eta^{1/2})n \leq (1+\delta+5\eta^{1/2})n, \]
such that
\[ |A \cap P| \geq (1-\eta^{1/2})n, \ \ |B \cap Q| \geq (1-\eta^{1/2})n. \]
This completes the proof.
\end{proof}

\section{Proof of Theorem \ref{carry}}\label{sec:carry}

In this section, we deduce Theorem~\ref{carry} from Theorem~\ref{1+theta}. The plan is to divide $A$ into its positive and negative parts, and apply Theorem~\ref{1+theta} to each of these two parts (with appropriately chosen $\Gamma$). After that, we need to ensure that the arithmetic progressions containing these two parts can be glued together so that they contain $0$ and are symmetric around $0$.

\begin{lemma}\label{counting}
If $A \subset \Z_{> 0}$ and $|A| = n$, then $C(A)< 1/2$.
\end{lemma}

\begin{proof}
Write $A = \{a_1 < a_2 < a_3 < \cdots < a_n\}$. For each $i$, the sums $a_i+a_1, \cdots, a_i+a_n$ are all of size greater than $a_i$, and thus at most $n-i$ of them can lie in $A$. It follows that
\[ C(A) \leq \frac{1}{n^2} \sum_{i=1}^n (n-i) < \frac{1}{2}.  \]
\end{proof}

\begin{lemma}\label{lem:carry-pos}
Let $\ee \in (0,2^{-20})$. Let $A \subset \Z_{> 0}$ be a set of $n \geq 2\ee^{-1/2}$ positive integers. If $C(A) \geq (1-\ee)/2$, then there is an arithmetic progression $P$ containing $0$ with size at most $(1 + 45\ee^{1/2})n$ such that $|A \cap P| \geq (1- \ee^{1/2})n$.
\end{lemma}

\begin{proof}
Let $\Gamma \subset A \times A$ be the subset defined by
\[ \Gamma = \{(a,a') \in A \times A \colon |a-a'| \in A\}. \]
Since $C(A) \geq (1-\ee)/2$, the number of pairs $(a,a') \in A \times A$ with $a-a' \in A$ is at least $(1-\ee)n^2/2$, and by symmetry there are also at least $(1-\ee)n^2/2$ pairs $(a,a') \in A \times A$ with $a'-a \in A$. It follows that $|\Gamma| \geq (1-\ee)n^2$. By the construction of $\Gamma$ we have
\[|A -_\Gamma A | \le 2n.\]
Since $\ee < 2^{-20}$, we have $2n \leq (1+\theta- 11\ee^{1/2})n$, where $\theta = \tfrac{\sqrt{5}+1}{2}$. Thus by Theorem~\ref{1+theta}, there is an arithmetic progression $P$ with 
\[ |P| \leq |A-_\Gamma A| - (1-5\ee^{1/2})n \leq (1+5\ee^{1/2})n, \]
such that $|A \cap P| \geq (1-\ee^{1/2})n$. It remains to show that $P$ can be extended to an arithmetic progression $\wt{P}$ containing $0$, and that $|\wt{P}\setminus P|$ is small.

Note that if $(a,a') \in \Gamma$, then either $a \in A\setminus P$, or $a' \in A\setminus P$, or $|a-a'| \in A\setminus P$, or $(a,a',|a-a'|) \in P \times P \times P$. Hence
\[  |\Gamma| \leq 4n |A\setminus P| + \#\{(a,a') \in P \times P \colon |a-a'| \in P\}. \]
Since $|\Gamma| \geq (1-\ee)n^2$ and $|A\setminus P| \leq \ee^{1/2}n$, we have
\[ \#\{(a,a') \in P \times P \colon |a-a'| \in P\} \geq (1-5\ee^{1/2})n^2. \]
It follows that the elements of $P$ must be in the residue class $0\pmod d$, where $d$ is the common difference of $P$, since otherwise $(P-P) \cap P= \emptyset$. Furthermore, from $|P| \leq (1+ 5\ee^{1/2})n$ we get
\[ \#\{(a,a') \in P \times P \colon |a-a'| \notin P\} \leq |P|^2 - (1-5\ee^{1/2})n^2 \leq 40\ee^{1/2}n|P|.  \]
Let $md$ be the smallest element in $P$, so that $P$ can be extended to an arithmetic progression $\wt{P}$ containing $0$, with $|\wt{P}| = |P| + m$. Then
\[ 
\begin{split}
\#\{(a,a') \in P \times P \colon |a-a'| \notin P\} &\geq \#\{(a,a') \in P \times P \colon |a-a'| < md\} \\
&\geq \min\left(m,|P|\right) |P|.
\end{split}
\]
Combining the upper and lower bounds above, we get $m \leq 40\ee^{1/2}n$, and hence
\[ |\wt{P}| = |P| + m \leq (1+5\ee^{1/2})n + 40\ee^{1/2}n \leq (1+45\ee^{1/2})n. \]
The proof is completed by replacing $P$ by $\wt{P}$.
\end{proof}

\begin{lemma}\label{lem:d1=d2}
Let $P_1,P_2$ be arithmetic progressions of the form
\[ P_1 = \{0,d_1,2d_1,\cdots,(\ell_1-1)d_1\}, \ \ P_2 = \{-(\ell_2-1)d_2,\cdots,-2d_2,-d_2,0\}. \]
If $d_1 \neq d_2$, then the number of pairs $(a_1,a_2) \in P_1 \times P_2$ with $a_1+a_2 \notin P_1\cup P_2$ is at least
\[ \tfrac{1}{4}\min(\ell_1^2, \ell_2^2) - \tfrac{1}{2}(\ell_1-\ell_2)^2 - \ell_1 - \ell_2.   \]
\end{lemma}

\begin{proof}
By dividing all elements in $P_1$ and $P_2$ by $(d_1,d_2)$, we may assume that $(d_1,d_2) = 1$. Consider the number of pairs $(a_1,a_2) \in P_1 \times P_2$ with $a_1+a_2 \in P_1$. This is the same as the number of solutions to $x_1d_1 - x_2d_2 = y_1d_1$, where $0 \leq x_1,y_1 < \ell_1$ and $0 \leq x_2 < \ell_2$. We must have $x_1 \equiv y_1 \pmod{d_2}$. Since $x_1-y_1$ must also be non-negative, the number of such solutions is at most
\[ \sum_{y_1 = 0}^{\ell_1-1} \left(\frac{\ell_1-1-y}{d_2} + 1\right) \leq \frac{\ell_1(\ell_1-1)}{2d_2} + \ell_1. \]
Similarly, the number of pairs $(a_1,a_2) \in P_1 \times P_2$ with $a_1+a_2 \in P_2$ is at most $\ell_2(\ell_2-1)/(2d_1) + \ell_2$. It follows that the number of pairs $(a_1,a_2) \in P_1 \times P_2$ with $a_1 + a_2 \notin P_1 \cup P_2$ is at least
\[ \ell_1\ell_2 - \frac{\ell_1(\ell_1-1)}{2d_2} - \frac{\ell_2(\ell_2-1)}{2d_1} - \ell_1 - \ell_2. \]
The conclusion follows by noting that at least one of $d_1,d_2$ is at least $2$.
\end{proof}

\begin{proof}[Proof of Theorem \ref{carry}]
Let $A_1$ and $A_2$ be the set of all positive and negative elements in $A$, respectively. Let $n_1= |A_1|$ and $n_2=|A_2|$. Then $n_1 +n_2 \ge n -1$. Let $X_i$ ($i \in \{1,2\}$) be the number of pairs $(a,a') \in A_i \times A_i$ with $a+a' \notin A_i$. By Lemma~\ref{counting} applied to $A_1$ and $A_2$,  we have $X_i \geq \tfrac{1}{2} n_i^2$. On the other hand, by assumption we have
\[ X_1 + X_2 \leq \left(\tfrac{1}{4} + \ee\right) n^2. \]
It follows that
\[ \left(\tfrac{1}{4}+\ee\right)n^2 \geq \tfrac{1}{2}(n_1^2+n_2^2) \ge \tfrac{1}{4}(n-1)^2 + \tfrac{1}{4}(n_1-n_2)^2,  \]
which implies that
\[ (n_1-n_2)^2 \leq 4\ee n^2 + 2n \leq 6 \ee n^2 \]
since $n \geq \ee^{-1}$. Hence $|n_1-n_2| \leq 3\ee^{1/2}n$, and thus both $n_1$ and $n_2$ lie in the range $[n/2 - 2\ee^{1/2}n, n/2 + 2\ee^{1/2}n]$. Furthermore, we have
\[
\begin{split} 
X_1 &\leq \left(\tfrac{1}{4}+\ee\right)n^2 - X_2 \leq \left(\tfrac{1}{4}+\ee\right)n^2 - \tfrac{1}{2}n_2^2 \\
&\leq \left(\tfrac{1}{4}+\ee\right)n^2 - \tfrac{1}{2}\left(\tfrac{n}{2}-2\ee^{1/2}n\right)^2 \leq \left(\tfrac{1}{8} + 2\ee^{1/2}\right) n^2 \\
&\leq \left(\tfrac{1}{2} + 16\ee^{1/2}\right) n_1^2, 
\end{split}
\]
and similarly we have
\[ X_2 \leq \left(\tfrac{1}{2} + 16\ee^{1/2}\right) n_2^2. \]
Since $\ee < 2^{-50}$, we may apply Lemma~\ref{lem:carry-pos} to $A_1$ and $-A_2$ to find arithmetic progressions $P_1,P_2$ containing $0$, such that
\[ |P_i| \leq (1+ 270\ee^{1/4})n_i, \ \ |A_i \cap P_i| \geq (1-6\ee^{1/4})n_i. \]
Let $d_1,d_2$ be the common differences of $P_1,P_2$, respectively. Consider
\[ T = \{(a_1,a_2) \in A_1 \times A_2 \colon a_1 + a_2 \notin A\} \]
and
\[ T' = \{(a_1,a_2) \in P_1 \times P_2 \colon a_1 + a_2 \notin P_1 \cup P_2\}. \]
If $(a_1,a_2) \in T'$, then either $(a_1,a_2) \in T$, or $a_1 \in P_1\setminus A_1$, or $a_2 \in P_2\setminus A_2$, or $a_1+a_2 \in A\setminus (P_1 \cup P_2)$.
\[ 
\begin{split}
|T'| &\leq |T| + |P_1\setminus A_1||P_2| + |P_2\setminus A_2||P_1| + |A\setminus (P_1\cup P_2)|\min(|P_1|,|P_2|)  \\
&\leq |T| + 276\ee^{1/4}n_1 \cdot 2n_2 + 276\ee^{1/4}n_2 \cdot 2n_1 + 6\ee^{1/4}n \cdot 2\min(n_1,n_2) \\
&\leq |T| + 282\ee^{1/4}n^2.
\end{split}
\]
From the assumption $C(A) \geq 3/4-\ee$ and the facts that $C(A_1) \leq 1/2$ and $C(A_2) \leq 1/2$, it follows that
\[ |T| \leq \left(\tfrac{1}{4}+\ee\right)n^2 - \tfrac{1}{2}n_1^2 - \tfrac{1}{2}n_2^2 \leq \ee n^2 + \tfrac{n}{2} \leq 2\ee n^2.   \]
Hence $|T'| \leq 284\ee^{1/4}n^2$. On the other hand, if $d_1 \neq d_2$ then we can apply Lemma~\ref{lem:d1=d2} to get
\[
|T'| \geq \tfrac{1}{4}\min(|P_1|^2, |P_2|^2) - \tfrac{1}{2}(|P_1|-|P_2|)^2 - |P_1| - |P_2|.
\]
Using the bounds
\[ |P_i| \leq (1+270 \ee^{1/4})n_i \leq \tfrac{1}{2} (1+ 280\ee^{1/4})n \leq n \]
and 
\[ |P_i| \geq (1-6\ee^{1/4})n_i \geq \tfrac{1}{2} (1-10\ee^{1/4}) n \geq \tfrac{1}{3}n,  \]
we conclude that
\[ |T'| \geq  \tfrac{1}{36}n^2 - \tfrac{1}{2} (290\ee^{1/4}n)^2 - 2n \geq \tfrac{1}{40}n^2. \]
This is a contradiction to the previously obtained upper bound $|T'| \leq 284\ee^{1/4}n^2$, and hence we must have $d_1 = d_2$. The conclusion follows by taking $P$ to be the union $P_1 \cup P_2$, 
so that
\[ |P| \leq 2\max(|P_1|,|P_2|) \leq 2(1+270\ee^{1/4})\max(n_1,n_2) \leq (1+ 280\ee^{1/4})n \]
and
\[ |A\cap P| = |A_1 \cap P_1| + |A_2 \cap P_2| \geq (1-6\ee^{1/4})(n_1+n_2) \geq (1-10\ee^{1/4})n. \]
\end{proof}

\section{A continuous analogue of Theorem~\ref{carry}}\label{sec:continuous-carry}

For a measurable set $A \subset \R$ with finite Lebesgue measure $\lambda(A)$, we study the following quantity related to $C(A)$:
\[ \langle 1_A*1_A, 1_A\rangle = \int_{\R} \int_{\R} 1_A(y)1_A(x-y)1_A(x) \d x \d y. \]

\begin{theorem}[ Riesz-Sobolev Inequality]
	Let $f,g,h$ be three nonnegative functions on $ \R^{n}$. Then with 
	\[I(f,g,h) =\int_{\R^{n}} \int_{\R^{n}} f(y)g(x-y)h(x) \d x \d y,  \]
we have $I(f,g,h) \le I(f^*, g^*, h^*)$, where $f^*, g^*, h^*$ are the corresponding symmetric - decreasing rearangement of functions $f, g ,h$.
\end{theorem}
One can find the proof in \cite[Chapter 3]{Lieb}.
To apply the Riez- Sobolev inequality, we notice that $A^*$, the symmetric rearangement of the set $A$ in our case is $I \subset \R$, the interval centered at the origin with $\lambda(I) = \lambda(A)$. Thus,
\begin{equation}\label{eq:rs} 
\langle 1_A*1_A, 1_A\rangle \leq \langle 1_I*1_I, 1_I\rangle = \tfrac{3}{4}\lambda(A)^2. 
\end{equation}
 Using Theorem~\ref{carry}, we may deduce the following structure theorem for sets $A$ when the inequality~\eqref{eq:rs} is almost sharp.

\begin{corollary}\label{cor:Riesz-Sobolev}
Let $A \subset \R$ be a measurable set with finite Lebesgue measure, and let $\ee \in (0,2^{-100})$. Assume that 
\[ \langle 1_A*1_A, 1_A\rangle \geq (1-\ee) \langle 1_I*1_I, 1_I\rangle, \]
where $I \subset \R$ is the interval centered at the origin with $\lambda(I) = \lambda(A)$. Then there exists an interval $J \subset \R$ centred at the origin such that $\lambda(J) \leq (1+ 561\ee^{1/4})\lambda(A)$ and $\lambda(A \cap J) \geq (1 - 21\ee^{1/4})\lambda(A)$.
\end{corollary}

A more general result concerning $\langle 1_A*1_B, 1_C\rangle$ for $A,B,C \subset \R$ was established in~\cite{Christ0,Christ1}, and even more generally in~\cite{Christ2,Christ3} for $A,B,C \subset \R^d$.

\begin{proof}
We construct the following discrete version $\CA \subset \Z$ of $A$ in the following standard way (c.f.~\cite[Section 6]{Christ0}). Let $\eta,\delta>0$ be sufficiently small in terms of $\ee$ and $\lambda(A)$. For each $n \in \Z$, let $I_n$ be the interval $(\eta n, \eta(n+1))$. Let
\[ \CA = \{n \in \Z \colon \lambda(A \cap I_n) \geq (1-\delta) \lambda(I_n)\}. \]
Let $\wt{A} = \cup_{n \in \CA} I_n$, so that $\lambda(\wt{A}) = \eta |\CA|$. By the Lebesgue differentiation theorem, we have $|A \triangle \wt{A}| \rightarrow 0$ as $\max(\eta,\delta) \rightarrow 0$. Thus by taking $\eta,\delta$ small enough, we may assume that
\[ \eta^{-1}\left(1-\tfrac{\ee}{100}\right)\lambda(A) \leq |\CA| \leq \eta^{-1} \left(1+\tfrac{\ee}{100}\right)\lambda(A), \]
and that
\[ \langle 1_{\wt{A}}*1_{\wt{A}}, 1_{\wt{A}} \rangle \geq \langle 1_A*1_A, 1_A \rangle - \tfrac{1}{100} \ee  \lambda(A)^2 \geq \left(\tfrac{3}{4} - \tfrac{4\ee}{5}\right) \lambda(A)^2. \]

Now we connect $\langle 1_{\wt{A}}*1_{\wt{A}}, 1_{\wt{A}} \rangle$ with $C(A)$. Note that
\[ \langle 1_{\wt{A}}*1_{\wt{A}}, 1_{\wt{A}} \rangle = \sum_{a,b,c \in \CA} \langle 1_{I_a}*1_{I_b}, 1_{I_c} \rangle.  \]
The summand above is nonzero only when $a+b = c$ or $a+b = c+1$, in which case it has value $\eta^2/2$. Hence
\[ \langle 1_{\wt{A}}*1_{\wt{A}}, 1_{\wt{A}} \rangle = \tfrac{1}{2} \eta^2 \left(\sum_{a,b,c \in \CA} 1_{a+b=c} + \sum_{a,b,c \in \CA} 1_{a+b=c+1}\right). \]
It follows that
\[ \sum_{a,b,c \in \CA} 1_{a+b=c} + \sum_{a,b,c \in \CA} 1_{a+b=c+1} \geq 2\eta^{-2}\left(\tfrac{3}{4} - \tfrac{4\ee}{5}\right) \lambda(A)^2 \geq \left(\tfrac{3}{2}-\tfrac{8\ee}{5}\right) |\CA|^2. \]

The first sum above is the number of summing triples in $\CA$, and the second sum above is the number of summing triples in $\CA - 1 = \{n-1 \colon n \in \CA\}$. Both of these two sums are bounded above by $\tfrac{3}{4} |\CA|^2 + O(|\CA|)$, and thus they are also both bounded below by
\[ \left(\tfrac{3}{2} - \tfrac{8\ee}{5}\right) |\CA|^2 - \tfrac{3}{4} |\CA|^2 - O(|\CA|) \geq \left(\tfrac{3}{4}-2\ee\right) |\CA|^2. \]
Hence $C(\CA) \geq 3/4 - 2\ee$ and $C(\CA-1) \geq 3/4 - 2\ee$.

By Theorem~\ref{carry}, there exist arithmetic progressions $P,P'$ centred at the origin such that
\[ |P| \leq (1+560\ee^{1/4})|\CA|, \ \ |P \cap \CA| \geq (1-20\ee^{1/4})|\CA|, \]
and
\[ |P'| \leq (1+560\ee^{1/4})|\CA|, \ \ |P' \cap (\CA-1)| \geq (1-20\ee^{1/4})|\CA|. \]
Since
\[
\begin{split} 
|P \cap (P'+1)| &\geq |P \cap \CA| - |\CA \setminus (P'+1)| \\
&\geq (1-20\ee^{1/4})|\CA| - 20\ee^{1/4}|\CA| \geq 0.99|\CA| \geq 0.9|P|,
\end{split}
\]
the common difference of $P$ and $P'$ must be $1$. Let $J$ be the smallest interval centred at the origin containing $I_n$ for all $n \in P$, so that
\[ \lambda(J) \leq \eta (|P|+1) \leq (1+560\ee^{1/4}) \eta |\CA| + \eta \leq (1+561\ee^{1/4})\lambda(A).  \] 
Moreover, we have
\[ \lambda(A \cap J) \geq \sum_{n \in \CA\cap P} \lambda(A \cap I_n).  \]
Each summand above is at least $(1-\delta)\lambda(I_n) = (1-\delta)\eta$ by the definition of $\CA$. Hence
\[ \lambda(A \cap J)  \geq (1-\delta)\eta |\CA \cap P| \geq (1-\delta)(1-20\ee^{1/4}) \eta |\CA| \geq (1-21\ee^{1/4})\lambda(A). \]
This completes the proof.

\end{proof}

\bibliographystyle{plain}
\bibliography{biblio}{}

\end{document}